\documentclass[a4paper]{article}

\usepackage[all]{xy}\usepackage[latin1]{inputenc}        
\usepackage[dvips]{graphics,graphicx}
\usepackage{amsfonts,amssymb,amsmath,color,mathrsfs, amstext}
\usepackage{amsbsy, amsopn, amscd, amsxtra, amsthm,authblk, enumerate}
\usepackage{upref}
\usepackage[colorlinks,
            linkcolor=red,
            anchorcolor=red,
            citecolor=red
            ]{hyperref}

\usepackage{geometry}
\usepackage{float}

\usepackage{yhmath}

\numberwithin{equation}{section}

\DeclareMathOperator{\supp}{supp}
\DeclareMathOperator{\dist}{dist}
\DeclareMathOperator{\var}{var}
\DeclareMathOperator{\tr}{tr}

\newcommand{\bbE}{\mathbb{E}}
\newcommand{\bbR}{\mathbb{R}}
\newcommand{\bbS}{\mathbb{S}}
\newcommand{\bbP}{\mathbb{P}}
\newcommand{\bbQ}{\mathbb{Q}}

\newtheorem{theorem}{Theorem}[section]
\newtheorem{lemma}{Lemma}[section]
\newtheorem{proposition}{Proposition}[section]
\newtheorem*{proposition*}{Proposition}

\newtheorem*{corollary*}{Corollary}

\newtheorem*{definitions*}{Definitions}
\newtheorem*{conjecture*}{\bf Conjecture}

\newtheorem*{example*}{\bf Example}
\theoremstyle{remark}
\newtheorem{remark}{\bf Remark}[section]

\newtheorem{assumption}{Assumption}[section]

\begin{document}

\title{Semi-groups of stochastic gradient descent and online principal component analysis: properties and diffusion approximations}

\author[1]{Yuanyuan Feng\thanks{yuanyuaf@andrew.cmu.edu}}
\author[2]{Lei Li\thanks{leili@math.duke.edu}}
\author[3]{Jian-Guo Liu\thanks{jliu@phy.duke.edu}}
\affil[1]{Department of Mathematics, Carnegie Mellon University, Pittsburgh, PA 15213, USA}
\affil[2]{Department of Mathematics, Duke University, Durham, NC 27708, USA.}
\affil[3]{ Departments of Mathematics and Physics, Duke University, Durham, NC 27708, USA.}

\date{}
\maketitle

\begin{abstract}
We study the Markov semigroups for two important algorithms from machine learning: stochastic gradient descent (SGD) and online principal component  analysis (PCA). We investigate the effects of small jumps on the properties of the semi-groups. Properties including regularity preserving, $L^{\infty}$ contraction are discussed. These semigroups are the dual of the semigroups for evolution of probability, while the latter are $L^{1}$ contracting and positivity preserving. Using these properties, we show that stochastic differential equations (SDEs) in $\bbR^d$ (on the sphere $\bbS^{d-1}$) can be used to approximate SGD (online PCA)  weakly. These SDEs may be used to provide some insights of the behaviors of these algorithms.
\end{abstract}

\section{Introduction}
 Stochastic gradient descent (SGD) is a stochastic approximation of the gradient descent optimization method for minimizing an objective function. It is widely used in support vector machines, logistic regression, graphical models and artificial neural networks, which shows amazing performance for large-scale learning due to its computational and statistical efficiency  \cite{bottou2010large,bubeck2015convex,bottou2016optimization}. Principal component  analysis (PCA) is a dimension reduction method which preserves most of the information in the large data set \cite{josse2011multiple}. Online PCA updates the current PCA each time new data are observed without recomputing it from scratch \cite{li2016near}. 
SGD and online PCA are both popular algorithms in machine learning. Computational efficiency and convegence behavior  in the context of large-scale learning  \cite{ghadimi2013stochastic,allen2016first}  of these two algorithms are studied tremendously.

In this paper, we focus on the properties of discrete semigroups for SGD and online PCA in Section \ref{sec:sgd} and Section \ref{sec:pca} respectively in the small jump regimes. Properties including regularity preserving, $L^{\infty}$ contraction are discussed. These semigroups are the dual of the semigroups for evolution of probability, while the latter are $L^{1}$ contracting and positivity preserving.  Based on these properties, we show that SGD and online PCA can be approximated in the weak sense by continuous-time stochastic differential equations (SDEs) in $\bbR^d$ or on the sphere $\bbS^{d-1}$  respectively.  These will help us understand the discrete algorithms in the viewpoint of diffusion approximation and randomly perturbed dynamical system. Other related works regarding diffusion approximation for SGD can be found in \cite{litaie2017,llql17}, while diffusion approximation using SDEs on sphere for online PCA seems new.

\section{The semigroups from SGD}\label{sec:sgd}
In machine learning, one optimization problem that appears frequently is
\begin{gather}\label{eq:optimization}
\min_{x\in \bbR^d}  f(x) ,
\end{gather}
where $f(x)$ is the loss function associated with a certain training set, and $d$ is the dimension for the parameter $x$. Usually, the training set is large and one instead considers the stochastic loss functions $f(x; \xi)$ such that
\begin{gather}
f(x)=\bbE f(x; \xi),
\end{gather}
and $f(x; \xi)$ is often much simpler (e.g. the loss function for a few randomly chosen samples) and thus much easier to handle. Here, $\xi \sim \nu$ is a random vector and $\nu$ is some probability distribution.  The stochastic gradient descent (SGD) is then to consider
\begin{gather}\label{eq:sgd}
x_{n+1}=x_n-\eta \nabla f(x_n; \xi_n) ,
\end{gather}
where $\xi_n\sim \nu$ are i.i.d  so that $\xi_n$ is independent of $x_n$, with the hope that $\{x_n\}$ can lead to some approximation (if not exact) solution to the optimization problem \eqref{eq:optimization}. 
Our goal in this section is to study the Markov chains formed by the SGD \eqref{eq:sgd}.

We introduce the following set of smooth functions
\begin{gather}\label{cknorm}
C_b^m(\bbR^d)=\left\{ f\in C^m(\bbR^d) ~~\Big|~~ \|f\|_{C^m}:=\sum_{|\alpha|\le m}|D^{\alpha}f|_{\infty}<\infty \right\}.
\end{gather}

\subsection{The semi-group and the properties}

Let $\bbE_{x_0}$ denote the expectation under the distribution of this Markov chain starting from $x_0$ and $\mu^n(\cdot; x_0)$ be the law of $x_n$.
Let $\mu(y, \cdot)$ be the transition probability. Then,
for any Borel set $E$, by the Markov property:
\begin{gather}\label{eq:markov}
\mu^{n+1}(E; x_0)=\int_{\bbR^d}\mu(y, E) \mu^n(dy; x_0)
=\int_{\bbR^d}  \mu^n(E;z) \mu(x_0, dz).
\end{gather}

For a fixed test function $\varphi\in L^{\infty}(\bbR^d)$, we define
\begin{gather}\label{eq:un}
u^{n}(x_0)=\bbE_{x_0}\varphi(x_n)=\int_{\bbR^d} \varphi(y) \mu^{n}(dy; x_0).
\end{gather}
The Markov property implies that
\begin{gather}
u^{n+1}(x_0)=\bbE_{x_0}(\bbE_{x_1}\varphi(x_{n+1})|x_1)
=\bbE_{x_0}u^n(x_1)=\int_{\bbR^d}\mu(x_0, dx_1) \int_{\bbR^d} \varphi(y)\mu^n(dy; x_1),
\end{gather}
which is consistent with \eqref{eq:markov}. Given the SGD \eqref{eq:sgd}, we find explicitly that
\begin{gather}\label{eq:expectation}
u^{n+1}(x)=\mathbb{E}(u^n(x-\eta \nabla f(x; \xi)))=:Su^n(x).
\end{gather}
Then, $u^0=\varphi$ and  $\{S^n\}_{n\ge 0}$ forms a semi-group for the Markov chain.

For the convenience of discussion, let us introduce $\dist(A, B)$ to mean the distance of two sets $A, B\subset \bbR^d$:
\[
\dist(A, B)=\inf_{x\in A, y\in B} |x-y|.
\]
We have the following claims regarding the effects of small jumps (small $\eta \nabla f(x, \xi)$ ) on the properties of semigroups:
\begin{theorem}\label{sequenceproperty}
Consider SGD \eqref{eq:sgd}. Let $u^n$ and $S$ be defined by \eqref{eq:un} and \eqref{eq:expectation} respectively. Then:
\begin{enumerate}[(i)]
\item (Regularity) Suppose that  for some $k\in \mathbb{N}$ $\varphi\in C^k(\bbR^d)$ and $\sup_{\xi}\|f(\cdot; \xi)\|_{C^{k+1}}<\infty$. Then  there exists $\eta_0>0$, such that 
\[
\|u^n\|_{C^k}\le C(k, T, \eta_0)\|\varphi\|_{C^k},~\forall~\eta\le \eta_0, ~n\eta\le T.
\]
\item ($L^{\infty}$ contraction) For any $n\ge 0$, $\|u^{n+1}\|_{L^\infty}=\|Su^{n}\|_{L^\infty}\le \|u^n\|_{L^\infty}$.

\item (Finite speed) If $\supp \varphi \subset K$ and $\sup_{\xi}\|f(\cdot; \xi)\|_{C^1}<\infty$, then for any $n\ge 0$, we have that $\dist(\supp u^n,K)\le CT$, where $C$ is a constant depending only on $f$.

\item (Mass confinement) Suppose $\sup_{\xi}\|f(\cdot; \xi)\|_{C^2}<\infty$ and that there exist $R>0, \delta>0$ such that whenever $|x|\ge R$, $\frac{x}{|x|}\cdot\nabla f(x; \xi)\ge \delta$ for any $\xi$. If $|x_0|\le R$, then for any $n\ge 0$, $\eta<\frac{2\delta R}{C^2}$, it holds that 
\[
\supp\mu^n\subset B(0,R+C\eta).
\]
\end{enumerate}
\end{theorem}

\begin{proof}
(i).  Since $u^{n+1}(x_0)=\bbE(u^n(x_0-\eta \nabla f(x_0; \xi)))$, for any $1\le i\le d$,  
\[
\partial_iu^{n+1}(x_0)=\bbE\left(\sum_{j=1}^d\partial_ju^n(x_0-\eta \nabla f(x_0; \xi)(\delta_{ij}-\eta\partial_{i}\partial_{j}f(x_0,\xi))\right).
\]
 By $\sup_{\xi}\|f(\cdot; \xi)\|_{C^{k+1}}<\infty$, we have  $ \|u^{n+1}\|_{C^1}\le (1+C\eta)\|u^{n}\|_{C^1}$.  Similar calculation reveals that for any $k$, there exits $\eta_0>0$ such that for any $\eta\le \eta_0$ there exists $C(k, \eta_0)$ satisfying
\[
 \|u^{n+1}\|_{C^k}\le (1+C(k,\eta_0)\eta)\|u^{n}\|_{C^k}.
\]
Since $n\eta\le T$,  we have  
\[
\|u^{n}\|_{C^k}\le  (1+C(k,\eta_0)\eta)^n\|\varphi\|_{C^1}\le e^{C(k,\eta_0)n\eta}\|\varphi\|_{C^k}\le e^{C(k,\eta_0)T}\|\varphi\|_{C^k}.
\]

(ii). That $\|u^{n+1}\|_{L^\infty}\le \|u^n\|_{L^\infty}$ is clear by \eqref{eq:expectation}. 

(iii).  By equation \eqref{eq:expectation}, $x\in\supp u^{n+1} \Leftrightarrow x-\eta \nabla f(x; \xi)\in \supp u^n$. Hence, with the assumption $\sup_{\xi}\|f(\cdot; \xi)\|_{C^1}<\infty$,
\[
\dist(\supp u^{n+1},\supp u^n)\le C\eta,
\]
which implies the claimed result.

(iv). If $|x_n|\le R$, then $|x_{n+1}|\le |x_n|+\eta|\nabla f(x_n,\xi_n)|\le R+C\eta$. If $|x_n|> R$, 
\[
|x_{n+1}|^2=|x_n|^2-2\eta x_n\cdot\nabla f(x_n;\xi_n)+\eta^2|\nabla f(x_n;\xi_n)|^2.
\]
By assumption, $|x_{n+1}|^2\le |x_n|^2-2\eta\delta|x_n|+C^2\eta^2$. If we take $\eta<\frac{2\delta R}{C^2}$, we obtain that $|x_{n+1}|^2\le |x_n|^2$. Hence we conclude that $|x_{n+1}|\le R+C\eta$ for any $n\ge 0$. 
\end{proof}

\begin{proposition}\label{dualsgd}
  Assume $\sup_{\xi}\|f(\cdot; \xi)\|_{C^2}<\infty$. If $\eta$ is sufficiently small, then there exists $S^*: L^1(\mathbb{R}^d) \to L^1(\mathbb{R^d})$ such that $S$ is the dual of $S^*$, and 
$S^*$ is given by 
\begin{gather}\label{eq:Sstar}
S^*\rho=\bbE \left( \frac{\rho}{|\det(I-\eta \nabla^2f)|}\circ h(\cdot ; \xi)\right),
\end{gather}
where $h(\cdot; \xi)$ is the inverse mapping of $x\mapsto x-\eta \nabla f(x; \xi)$ and $'\circ'$ means function composition. Further, $S^*$ satisfies:
\begin{enumerate}[(i)]
\item  $\int_{\bbR^d} S^*\rho dx=\int_{\bbR^d} \rho dx$.

\item 
If $\rho\in L^1$ is nonnegative, then $S^*\rho\ge 0$ and $\|S^*\rho\|_1=\|\rho\|_1$. 
 $S^*$ is $L^1$ contraction.
 \end{enumerate}
\end{proposition}

\begin{proof}
From the expression, we see directly that
\[
|S^*\rho|\le \bbE \Big|\left( \frac{\rho}{|\det(I-\eta \nabla^2f)|}\circ h(\cdot ; \xi)\right)\Big|
=\bbE \left( \frac{|\rho|}{\det(I-\eta \nabla^2f)}\circ h(\cdot ; \xi)\right).
\]
This inequality implies that if $\rho\in L^1$, then $S^*\rho\in L^1$.

Take $\rho \in L^1$. Since $Su(x)=\mathbb{E}(u(x-\eta \nabla f(x;\xi)))$ for $u\in L^{\infty}$, we find that
\begin{gather*}
\langle Su(x), \rho \rangle
=\int_{\bbR^d} \mathbb{E}(u(x-\eta \nabla f(x;\xi))) \rho(x)\,dx
=\bbE \int_{\bbR^d} u(x-\eta \nabla f(x;\xi)) \rho(x)\,dx.
\end{gather*}
When $\eta$ is small, $x-\eta \nabla f(x; \xi)$ is bijective  for each $\xi$ because $\|f(\cdot; \xi)\|_{C^2}$ is uniformly bounded. Denote $h(y; \xi)$ the inverse mapping of $y=x-\eta\nabla f(x; \xi)$. It follows that
\[
\langle Su(x), \rho \rangle=\bbE \int_{\bbR^d}  u(y)  \frac{\rho}{|\det(I-\eta \nabla^2f)|}\circ h(y; \xi) dy .
\]
Using the fact that $(L^1)'=L^{\infty}$, we conclude that $S$ is the dual operator of $S^*$.

$(i)$  Take $u\equiv 1$, we have $Su\equiv 1$. In this case, $\int_{\bbR^d} S^*\rho\, dx=\int_{\bbR^d}\rho Su\, dx=\int_{\bbR^d}\rho\, dx$.

$(ii)$ That $\rho \ge 0$ implies that $S^*\rho\ge 0$ is obvious by the expression. Using Crandall-Tartar lemma \cite[Proposition 1]{crandalltarar80},  we get $\|S^*\rho\|_1\le\|\rho\|_1$ for general $\rho\in L^1$.

\begin{remark}
We remark that \eqref{eq:Sstar} is consistent with the first equality in \eqref{eq:markov}. To see this, assume $\mu^n$ is absolutely continuous to Lebesgue measure so that $\rho^n(\cdot ; x_0)=\frac{d\mu^n(\cdot; x_0)}{dx}$.
Then, \eqref{eq:markov} implies that
\[
\rho^{n+1}(x; x_0)=\bbE \int_{\bbR^d} \rho^n(y; x_0) \delta(x-(y-\eta \nabla f(y ;\xi)))\, dy=S^*\rho^n(x; x_0).
\]
\end{remark}

\end{proof}

\subsection{The diffusion approximation}

The discrete semi-groups are close to the continuous semi-groups generated by certain SDEs in the weak sense. Consider the SDE in It\^o sense \cite{oksendal03} given by
\begin{gather}\label{eq:sde1}
dX=b(X)\, dt+\sqrt{\eta\Sigma}\, dW .
\end{gather}
We use $e^{tL}\varphi$ to represent the solution to the backward Kolmogrov equation
\[
u_t=Lu:=b(x)\cdot\nabla u+\frac{1}{2}\eta \Sigma:\nabla^2 u, ~u(\cdot, 0)=\varphi.
\]
It is well-known that \cite{oksendal03}
\begin{gather}\label{eq:u}
u(x, t)=e^{tL}\varphi=\bbE_x\varphi(X(t)).
\end{gather}
Similarly, we use $e^{tL^*}\rho_0$ to represent the solution of the forward Kolmogrov equation (Fokker-Planck equation) at time $t$:
\[
\rho_t=L^*\rho:=-\nabla\cdot(b\rho)+\frac{1}{2}\eta \sum_{i,j}\partial_{ij}(\Sigma_{ij}\rho),~\rho(\cdot, 0)=\rho_0.
\]
Then, $\{e^{tL}\}$ and $\{e^{tL^*}\}$ form two semi-groups. If $\rho_0$ is the initial distribution of $X(t)$, then $e^{tL^*}\rho_0$ is the probability distribution at time $t$. 

\begin{lemma}
\begin{enumerate}[(i)]
\item
If $\rho_0\in L^1(\mathbb{R}^d)$ and $\rho_0\ge 0$, then $e^{tL^*}\rho_0\ge 0$ and $\|e^{tL^*}\rho_0\|_1=\|\rho_0\|_1$. Consequently, $e^{tL}$ is $L^1$-contraction, i.e. $\forall \rho_0\in L^1$, 
\[
\|e^{tL^*}\rho_0\|_1\le \|\rho_0\|_1,~\forall t\ge 0.
\]
\item 
$e^{tL}: L^{\infty}(\bbR^d)\to L^{\infty}(\bbR^d)$ is a contraction.
\end{enumerate}
\end{lemma}
\begin{proof}
(i). Since $e^{tL^*}$ is the solution operator to Fokker-Planck equations, it is well-known that it preserves positivity and probability \cite{soize1994fokker}, and for general initial data $\rho_0\in L^1$, 
\[
\int_{\bbR^d} e^{tL^*}\rho_0\, dx=\int_{\bbR^d}\rho_0\, dx.
\]
 Then, $\forall u, v\in L^1$ and $u\le v$, it holds $e^{tL^*}u\le e^{tL^*}v$, and consequently 
 \[
 \|e^{tL^*}\rho_0\|_1\le \|\rho_0\|_1,~~\forall \rho_0\in L^1,
 \]
  by Crandall-Tartar lemma \cite[Proposition 1]{crandalltarar80}.

(ii). Fix $\varphi \in L^\infty$. We take $\rho \in L^1$, and have 
\[
\langle e^{tL}\varphi,\rho\rangle=\langle \varphi, e^{tL^*}\rho\rangle\le\|\varphi\|_{L^\infty}\|e^{tL^*}\rho\|_{L^1}\le\|\varphi\|_{L^\infty}\|\rho\|_{L^1},
\] which yields that $\|e^{tL}\varphi\|_{L^\infty}\le\|\varphi\|_{L^\infty}$. 
\end{proof}

We now show that the discrete semi-groups can be approximated by the continuous semi-groups in the weak sense (this is the standard terminology in SDE analysis while in functional analysis a more appropriate term might be `weak-star sense' ):
\begin{theorem}\label{thm:sdeforsgd}
Assume that $\sup_{\xi}\|f(\cdot,\xi)\|_{C^5}<\infty$. Consider the SDE \eqref{eq:sde1} in It\^o sense. $u^n(x)$ and $u(x, t)$ are given in \eqref{eq:un} and \eqref{eq:u} respectively.
If we choose $b(x)=-\nabla f(x)$ and $\Sigma\in C_b^2(\mathbb{R}^d)$ be positive semi-definite, then for all $\varphi\in C_b^4(\mathbb{R}^d)$, there exist $\eta_0>0$ and $C(T,  \|\varphi\|_{C^4},\eta_0)>0$ such that
\[
	\sup_{n: n\eta\le T}\|u^n-u(\cdot, n\eta)\|_{\infty}\le C(T,  \|\varphi\|_{C^4},\eta_0)\eta,
	~\forall \eta\le \eta_0.
\]
If instead $\sup_{\xi}\|f(\cdot,\xi)\|_{C^7}<\infty$ and we choose $b(x)=-\nabla f(x)-\frac{1}{4}\eta\nabla|\nabla f(x)|^2$, $\Sigma=\var(\nabla f(x;\xi))$, then for all $\varphi\in C_b^6(\mathbb{R}^d)$, there exist $\eta_0>0$ and $C(T,  \|\varphi\|_{C^6},\eta_0)>0$ such that
\[
\sup_{n: n\eta\le T}\|u^n-u(\cdot, n\eta)\|_{\infty}\le C(T, \|\varphi\|_{C^6},\eta_0)\eta^2,
~\forall \eta\le \eta_0.
\]
\end{theorem}
\begin{proof}
First of all, we have
\begin{equation}\label{con}
	u(x, (n+1)\eta)=e^{\eta L}u(x,n\eta),~\forall n\ge 0.
\end{equation}

Since $\sup_{\xi}\|f(\cdot,\xi)\|_{C^5}<\infty$, for any $\varphi\in C_b^4(\mathbb{R}^d)$, there exists $\eta_0>0$ such that we have the semi-group expansion by Theorem \ref{sequenceproperty},
\[
|e^{\eta L}u^n(x)-u^n(x)-\eta Lu^n(x)|\le C(\|u^n\|_{C^4})\eta^2\le C(T,\|\varphi\|_{C^4}, \eta_0)\eta^2,~\forall \eta\le \eta_0.
\]
We therefore have
\begin{gather}\label{semiexpan1}
\Big|e^{\eta L}u^n(x)-u^n(x)-\eta b\cdot \nabla u^n(x)-\frac{1}{2}\eta^2\Sigma:\nabla^2u^n(x) \Big|\le C(T,\|\varphi\|_{C^4},\eta_0)\eta^2.
\end{gather}

If $\sup_{\xi}\|f(\cdot,\xi)\|_{C^7}<\infty$ and we take $\varphi\in C_b^6(\mathbb{R}^d)$, then there exists $\eta_0>0$ and we have for $\eta\le \eta_0$ by Theorem \ref{sequenceproperty}:
\[
|e^{\eta L}u^n(x)-u^n(x)-\eta Lu^n(x)-\frac{1}{2}\eta^2L^2u^n(x)|\le C(\|u^n\|_{C^6},\eta_0)\eta^3\le C(T,  \|\varphi\|_{C^6}, \eta_0)\eta^3.
\]
Therefore, we have
\begin{multline}\label{semiexpan}
\Big|e^{\eta L}u^n(x)-u^n(x)-\eta b\cdot \nabla u^n(x)-\frac{1}{2}\eta^2(\Sigma+bb^T):\nabla^2u^n(x) \\
-\frac{1}{4}\eta^2\nabla |b|^2\cdot\nabla u^n(x)\Big|\le C(T,\|\varphi\|_{C^6},\eta_0)\eta^3.
\end{multline}

On the other hand, we use Taylor expansion to \eqref{eq:expectation},
	\begin{equation}\label{taylor}
	|u^{n+1}(x)-u^n(x)+\eta\nabla f(x)\cdot\nabla u^n(x)-\frac{1}{2}\eta^2\bbE(\nabla f(x;\xi)\nabla f(x;\xi)^T):\nabla^2u^n(x)|\le C\eta^3.
	\end{equation}
If we choose $b(x)=-\nabla f(x)$, \eqref{semiexpan1} and \eqref{taylor} imply that
\begin{gather}\label{dis}
R^n:=\|u^{n+1}-e^{\eta L}u^n(x)\|_{\infty}\le 
\frac{1}{4}\eta^2\|\nabla |b|^2 \nabla u^n\|_{\infty}+C(T,\|\varphi\|_{C^4},\eta_0)\eta^2
\le C \eta^2.
\end{gather}
If instead we choose $b(x)=-\nabla f(x)-\frac{1}{4}\eta\nabla|\nabla f(x)|^2$ and $\Sigma=\var(\nabla f(x;\xi))$,  then \eqref{semiexpan} and \eqref{taylor} imply that
\begin{gather}
R^n:=\|u^{n+1}-e^{\eta L}u^n(x)\|_{\infty}\le C(T,\|\varphi\|_{C^6},\eta_0)\eta^3.
\end{gather}
We define $E^n:=\|u^n-u(\cdot,n\eta)\|_{L^\infty}$, \eqref{con} and the definition of $R^n$ yield that
\[
E^{n+1}\le  \left\|e^{\eta L}\Big(u(\cdot, n\eta)-u^n\Big)\right\|_{\infty}+R^n\le \|u(\cdot, n\eta)-u^n\|_{L^\infty}+R^n=E^n+R^n.
\]
The second inequallity holds because $e^{tL}$ is $L^\infty$ contraction. 
The result then follows.
\end{proof}

\begin{remark}
To get the $O(\eta)$ weak approximation, we can even take $\Sigma=0$. Indeed, with $\Sigma=0$, $Z=X(n\eta)-x_n$ is a noise with magnitude $O(\sqrt{\eta})$. The weak order is $O(\eta)$ because $\mathbb{E}Z=O(\eta)$. Choosing $\Sigma=\var(\nabla f(x;\xi))$ does not improve the weak order from $O(\eta)$ to $O(\eta^2)$ but we believe it characterizes the leading order fluctuation which is left for future investigation.
\end{remark}

\subsection{A specific example}

In learning a deep neural network with $N\gg 1$ training samples, the loss function is often given by
\[
f(x)=\frac{1}{N}\sum_{k=1}^N f_k(x).
\] 
To train a neural network, the back propogation algorithm is often applied to compute $\nabla f_k(x)$, which is usually not trivial, making computing $\nabla f(x)$ expensive. One strategy is to pick $\zeta\in \{1, \ldots, N\}$ uniformly and the following SGD is applied
\[
x_{n+1}=x_n-\eta \nabla f_{\zeta}(x_n).
\]
Such  SGD algorithm and its SDE approximation were studied in \cite{litaie2017}. The results in Theorem \ref{thm:sdeforsgd} can be viewed as a slight generalization of the results in \cite{litaie2017}, but our proof is performed in a clear way based on the semi-groups.  The operators $S$ and $\Sigma$ for this specific example are given respectively by
\begin{gather*}
Su(x)=\frac{1}{N}\sum_{k=1}^N u(x-\eta \nabla f_k(x)), ~~ \Sigma=\frac{1}{N}\sum_{k=1}^N (\nabla f(x)-\nabla f_k(x))\otimes (\nabla f(x)-\nabla f_k(x)).
\end{gather*}

Now that we have the connection between the SGD and SDEs, the well-known results in SDEs can be borrowed to understand the behaviors of SGD. For example, this gives us the intuition how the batch size in a general SGD can help to escape from sharp minimizers and saddle points of the loss function by affecting the diffusion (see \cite{llql17}).

\section{The semigroups from online PCA}\label{sec:pca}

Assume some data have $d$ coordinates and they can be represented by points in $\bbR^d$. Let  $\xi \in \bbR^d$ be a random vector sampled from the distribution of the data with mean centered to zero (if not, we use $\xi-\bbE \xi$) and the covariance matrix to be
\[
\Sigma=\bbE(\xi\xi^T).
\] 
Principal component analysis (PCA) is a procedure to find $k$ ($k< d$) linear combinations of these coordinates: $w_1^T\xi, \ldots, w_k^T\xi$ such that for all $i=1, 2,\ldots, k$, the optimization problem is solved
\begin{gather}
\max\bbE (w_i^T\xi)^2,~\text{subject to } w_i^T w_j=\delta_{ij},~j\le i.
\end{gather} 
It is well-known that $w_1, \ldots, w_k$ are the eigenvectors of $\Sigma$ corresponding to the first $k$ eigenvalues.

In the online PCA procedure, an adaptive system receives a stream of data $\xi(n)\in \bbR^d$ and tries to compute the estimates of $w_1, \ldots, w_k$ \cite{oja82,oja92}. 
The algorithm in \cite{oja82} in the case $k=1$ can be summarized as 
\begin{gather}
w^{n}=\bbQ\left(w^{n-1}+\eta \nabla f(w^{n-1}; \xi(n))\right)=\frac{w^{n-1}+\eta(n)\xi(n)\xi(n)^{T}w^{n-1}}{|w^{n-1}+\eta(n)\xi(n)\xi(n)^{T}w^{n-1}|},
\end{gather}
where $f(w; \xi(n))=(w^T\xi(n))^2$ and 
\[
\mathbb{Q}v:=v/|v|.
\]
This algorithm is also called the stochastic gradient ascent (SGA) algorithm.

Suppose we choose $\eta(n)=\eta$ to be a constant and  take
\begin{gather*}
\xi(n)\sim  \nu,~i.i.d.,
\end{gather*}
where $\nu$ is some probability distribution in $\bbR^d$. Then, $\{w^n\}$ forms a time-homogeneous Markov chain on $\bbS^{d-1}$. Our goal in this section is to study this Markov chain and its semi-groups. 

For the convenience of following discussion,  we assume:
\begin{assumption}\label{ass:boundednessofxi}
There exists $C>0$ such that $\forall \xi \sim \nu$,
\[
|\xi |\le C.
\]
\end{assumption}

\subsection{The semi-groups and properties}

Similarly as in Section \ref{sec:sgd}, fix a test function $\varphi\in L^{\infty}(\bbS^{d-1})$, and define
\begin{gather}\label{eq:sphereun}
u^n(w^0)=\bbE_{w^0}\varphi(w^n)=\int_{\bbS^{d-1}}\varphi(y)\mu^n(dS; w^0) .
\end{gather}
Again by the Markov property, for the SGA algorithm, we have
\begin{gather}\label{pcaiter}
u^{n+1}(w)=\bbE u^n(\bbQ(w+\eta \xi\xi^Tw))=:S u^n (w),
\end{gather}
with $u^0(w)=\varphi(w)$. Clearly, $\{S^n\}_{n\ge 0}$ form a semi-group.

For discussing the dynamics on sphere, we find it is convenient to extend $w$ into a neighborhood of the sphere as 
\[
w=\frac{x}{|x|},~x\in \bbR^d,
\]
and introduce the projection:  
\begin{gather}
 \bbP :=I-w\otimes w.
\end{gather}
This extension allows us to perform the computation on sphere by performing computation in $\bbR^d$. For example, if $\psi\in C^1(\bbS^{d-1})$, then we extend $\psi$ into a neighborhood as well and the gradient operator on the sphere can be written in terms of this extension as:
\begin{gather}
\nabla_{S} \psi(w)=\bbP\nabla \psi(x)|_{x=w}=(I-w\otimes w)\cdot\nabla \psi(w),~w\in \bbS^{d-1}.
\end{gather}
Clearly, $\nabla_S\psi$ only depends on the values of $\psi$ on $\bbS^{d-1}$, not on the extension. The extension is introduced for the convenience of computation. We use $C^k(\bbS^{d-1})$ to denote the space of functions that are $k$-th order continuously differentiable on the sphere with respect to $\nabla_S$, with the norm:
\begin{gather}
\|\psi\|_{C^k(\bbS^{d-1})}:= \sum_{|\alpha|\le k} \sup_{w\in \bbS^{d-1}}|\nabla_S^{\alpha} \psi(w)|.
\end{gather}
Here $\alpha=(\alpha_1, \ldots, \alpha_m)$ , $m\le k$ is a multi-index so that for a vector $v=(v^1, v^2, \ldots, v^d)$, $v^{\alpha}=\prod_{j=1}^mv^{\alpha_j}$.

Now we study the semi-groups for the online PCA:
\begin{proposition}
\begin{enumerate}[(i)]
\item $S: L^{\infty}(\bbS^{d-1})\to L^{\infty}(\bbS^{d-1})$ is a contraction. 
\item For any $\varphi\in C^k(\bbS^{d-1})$, there exists $\eta_0>0$ and $C(k, \eta_0, T)$ such that
\[
\|u^n\|_{C^k(\bbS^{d-1})}\le C(k,T)\|\varphi\|_{C^k(\bbS^{d-1})},~\forall \eta\le \eta_0.
\]

\item There exists $S^*: L^1(\bbS^{d-1})\to L^1(\bbS^{d-1})$ such that $S$ is the dual of $S^*$. $S^*$ is a contraction on $L^1$. Further, for any $\rho\in L^1$ and $\rho\ge 0$, we have $S^*\rho\ge 0$ and $\|S^*\rho\|_1=\|\rho\|_1$.
\end{enumerate}
\end{proposition}
\begin{proof}
(i). That $S$ is an $L^{\infty}$ contraction is clear from \eqref{pcaiter}.

(ii). For each $k$, there exists $\eta_0>0$ and $C(k, T, \eta_0)>0$ such that
\[
\|\bbQ(w+\eta \xi \xi^T w)\|_{C^k(\bbS^{d-1}; \bbS^{d-1})}
\le 1+C\eta,~\eta\le \eta_0.
\]
By \eqref{pcaiter} and chain rule for derivatives on sphere, we have 
\begin{gather}
\|u^{n+1}\|_{C^k(\bbS^{d-1})}\le  (1+C_1\eta)\| u^n\|_{C^k(\bbS^{d-1})},~\eta\le\eta_0,
\end{gather}
and the second claim follows.

(iii). Similar to Proposition \ref{dualsgd}, we find that $S^*$ is given by 
\[
S^*\rho(v)=\bbE(\rho\big(h(v; \xi)\big)|J_h(v; \xi))|),
\]
 where $h(v;\xi)=\frac{(I+\eta\xi\xi^T)^{-1}v}{|(I+\eta\xi\xi^T)^{-1}v|}$ is the inverse mapping of $\bbQ(w+\eta \xi\xi^Tw)$, and $|J_h|$ accounts for the volume change on $\bbS^{d-1}$ under $h$. 
 The properties are then similarly proved as in Proposition  \ref{dualsgd}.
\end{proof}

\subsection{The diffusion approximation}
Now, we move onto the SDE approximation to the online PCA, i.e. seeking a semi-group generated by diffusion processes on sphere to approximate the discrete semi-groups. 

Before the discussion, let's consider the second order tensor $\nabla_{S}^2u$
\begin{gather}
\nabla_{S}^2u=\nabla_S(\nabla_S u).
\end{gather}
 Recall that for a vector field $\phi\in C^2(\bbS^{d-1}; \bbR^d)$, the divergence of $\phi$ is defined as
\[
\int_{\bbS^{d-1}} \mathrm{div}(\phi)\,\varphi\, dS=-\int_{\bbS^{d-1}} \phi\cdot \nabla_{S}\varphi\,dS.
\]
Again we extend the vector field into a neighborhood of the sphere so that we can use formulas in $\bbR^d$ to compute. Using the formula $\int_{\bbS^{d-1}} \nabla_S\cdot \phi dS=\int_{\bbS^{d-1}}(\nabla\cdot w) w\cdot \phi dS=(d-1)\int_{\bbS^{d-1}} w\cdot \phi dS$, one can derive that 
\begin{gather}
\mathrm{div}(\phi)=\nabla_S\cdot (\bbP \phi).
\end{gather} 
It follows that
\begin{gather}
\tr(\nabla_S^2u)=\nabla_S\cdot(\nabla_Su)=\mathrm{div}(\nabla_S u)=:\Delta_S u,
\end{gather}
because $\bbP\nabla_Su=\nabla_Su$.  $\Delta_Su$ is called the Laplace-Beltrami operator on $\bbS^{d-1}$.

Next we give the Taylor expansion and the proof can be found in Appendix \ref{app:proofspheretaylor}.
\begin{lemma}\label{lmm:pcaspheretaylor}
Let $w\in \bbS^{d-1}$, and $v\in \bbR^d$. Suppose $u\in C^3(\bbS^{d-1})$, then we have
\[
u\left(\frac{w+\eta v}{|w+\eta v|}\right)
=u(w)+\eta v\cdot\nabla_{S}u+\frac{1}{2}\eta^2 (vv:\nabla_S^2u-w\cdot v\otimes v\cdot\nabla_S u)+R(\eta)\eta^3,
\]
where $R(\eta)$ is a bounded function.
\end{lemma}

Now, we introduce $f(w; \xi)=\frac{1}{2}w^T\xi\xi^Tw$ and thus
\begin{gather}
f(w)=\bbE f(w; \xi)=\frac{1}{2}w^T\Sigma w.
\end{gather}
Define the following second moment:
\begin{gather}
M(w):=\bbE \nabla f(w; \xi) \nabla f(w, \xi)^T=w\otimes w : \bbE\, (\xi\otimes\xi\otimes \xi\otimes\xi ).
\end{gather}
By Lemma \ref{lmm:pcaspheretaylor}, we then have:
\begin{multline}\label{eq:spheretay}
u^{n+1}(w)
=u^n(w)+\eta \nabla_S f(w)\cdot \nabla_S u^n
+\frac{1}{2}\eta^2(M(w):\nabla_S^2 u^n-w\cdot M(w)\cdot \nabla_S u^n)+R(\eta)\eta^3.
\end{multline}

We now construct SDEs on $\bbS^{d-1}$ whose generator approximates the right hand side of \eqref{eq:spheretay}. Consider the following general SDE in Stratonovich sense in $\mathbb{R}^d$:
\begin{gather}\label{eq:sde2}
dX=\bbP b(X)\, dt+\bbP\sigma(X) \circ dW,
\end{gather}
where $W$ is the standard Wiener process (Brownian motion) in $\mathbb{R}^d$ while $'\circ'$ here represents Stratonovich stochastic integrals, convenient for SDEs on manifold \cite{hsu02}.

\begin{lemma}
$X$ stays on $\bbS^{d-1}$ if $X(0)\in \bbS^{d-1}$. In other words, $|X(t)|=1$.
\end{lemma}
To see this, we only need to apply It\^o's formula (for Stratonovich integrals) with the test function $f(x)=|x|^2$ and noting $\bbP \nabla f=\nabla_S f=0$.

Let $\varphi\in C^3(\bbS^{d-1})$ and we extend it to the ambient space of $\bbS^{d-1}$. Consider
\begin{gather}\label{eq:usphere}
u(x, t)=\bbE_x \varphi(X).
\end{gather}
We find using It\^o's formula (for Stratonovich integrals) that (the explicit expression of $\nabla_S^2$ in Appendix \ref{app:proofspheretaylor} is needed to derive this)
\begin{gather}\label{eq:spherebackkolmogrov}
u_t =(\bbP b+\bbP b_1(\sigma))\cdot\nabla_Su+\frac{1}{2}\sigma\sigma^T:\nabla_S^2u =: L_Su,
\end{gather}
where
\begin{gather}\label{eq:pb1}
(\bbP b_1(\sigma))_i=\frac{1}{2}(\bbP\sigma)_{kj}(\bbP\partial_k \sigma)_{ij}.
\end{gather}
The operator $L_S$ is elliptic on sphere. 
\begin{remark}
Clearly, if we take $\sigma=I$, then we have
\[
u_t=\frac{1}{2}\Delta_Su.
\]
This means that $dX=P\circ dW$ gives the spherical Brownian motion, which is the Stroock's representation of spherical Brownian motion \cite[Section 3.3]{hsu02}.
\end{remark}

With \eqref{eq:spheretay} and \eqref{eq:spherebackkolmogrov}, we conclude that the discrete semi-groups for online PCA can be approximated in the weak sense by the stochastic processes on the sphere:
\begin{theorem}
Assume Assumption \ref{ass:boundednessofxi}. 
Consider the SDE \eqref{eq:sde2} with $\sigma=:\sqrt{\eta}  S$:
\begin{gather}\label{eq:SDE3}
dw=\bbP b(w)\, dt+\sqrt{\eta}\bbP S(w)\circ dW.
\end{gather}
 $u^n$ and $u(w, t)$ are given as in \eqref{eq:sphereun} and \eqref{eq:usphere} respectively.  If we take $\bbP b(w)=\nabla_S f(w)=(I-w\otimes w)\cdot \Sigma w$ and $S(\cdot)\in C_b^2(\mathbb{S}^{d-1})$ to be positive semi-definite, then  for all $\varphi\in C_b^4(\mathbb{S}^{d-1})$, there exist $\eta_0>0$ and $C(T, \|\varphi\|_{C^4},\eta_0)>0$ such that $\forall \eta\le \eta_0$,
\[
   \sup_{n: n\eta\le T}\|u^n-u(\cdot, n\eta)\|_{L^{\infty}(\bbS^{d-1})}\le C(T, \|\varphi\|_{C^4}, \eta_0)\eta.
\]
If instead we take
\begin{gather*}
\begin{split}
& \bbP b(w)=\nabla_Sf(w)-\frac{1}{2}\eta \bbP(S(w)^2\cdot w)
-\frac{1}{2}\eta \bbP b_1(S)(w)
-\frac{1}{2}\eta \nabla_Sf(w)\cdot \nabla_S^2 f(w),\\
& S(w)=\sqrt{\var(\nabla f(w; \xi))}=\sqrt{M(w)-\nabla_Sf(w) \nabla_Sf(w)^T }.
\end{split}
\end{gather*}
where $Pb_1(\cdot)$ is given in \eqref{eq:pb1}, then for all  $\varphi\in C_b^6(\mathbb{S}^{d-1})$, there exist $\eta_0>0$ and $C(T, \|\varphi\|_{C^6}, \eta_0)>0$ such that $\forall \eta\le \eta_0$,
\[
   \sup_{n: n\eta\le T}\|u^n-u(\cdot, n\eta)\|_{L^{\infty}(\bbS^{d-1})}\le  C(T,  \|\varphi\|_{C^6}, \eta_0) \eta^2.
\]
\end{theorem}
The proof is similar as that for Theorem \ref{thm:sdeforsgd} and we omit here for brevity.

\section*{Acknowledgements}
The work of J.-G Liu is partially supported by KI-Net NSF RNMS11-07444 and NSF DMS-1514826. Y. Feng is supported by NSF DMS-1252912.

\appendix

\section{Proof of Lemma \ref{lmm:pcaspheretaylor}}\label{app:proofspheretaylor}

\begin{proof}[Proof of Lemma \ref{lmm:pcaspheretaylor}]
We set 
\[
g(\eta)=u\left(\frac{w+\eta v}{|w+\eta v|}\right) .
\]
By the fact that $|w|=1$ and direct computation, we have
\begin{gather*}
g'(0)=\nabla u \cdot (I-w\otimes w)\cdot v=v\cdot\nabla_S u,
\end{gather*}
and that
\begin{gather*}
g''(0)=(v\cdot(I-ww))\cdot\nabla^2u\cdot((I-ww)\cdot v)
+\nabla u\cdot(-2v (w\cdot v)-w(v\cdot v)+3w(w\cdot v)^2).
\end{gather*}
Since $w=x/|x|$, we have for $x=w\in \bbS^{d-1}$ that 
\[
\nabla w=I-w\otimes w.
\]
It follows that
\[
\nabla_S^2u(w)=(I-w\otimes w)\cdot\nabla^2u(w)\cdot(I-w\otimes w)
-[(w\cdot\nabla u)(I-w\otimes w)+((I-w\otimes w)\cdot\nabla u)\otimes w].
\]
Hence, we find that
\[
g''(0)=v\otimes v:\nabla_{S}^2u+\nabla u\cdot (w(w\cdot v)^2-v(w\cdot v))
=v\otimes v:\nabla_{S}^2u-w\cdot v\otimes v\cdot\nabla_S u,
\]
and $g'''(\eta)$ is bounded by the assumption. Hence, the claim follows from Taylor expansion.
\end{proof}

\bibliographystyle{plain}
\bibliography{sdealg}

\end{document}